\documentclass{amsart}

\usepackage[english]{babel}
\usepackage[utf8x]{inputenc}
\usepackage{amsmath,amsthm,amssymb,amsfonts}
\usepackage{verbatim}
\usepackage{graphicx}
\usepackage[colorinlistoftodos]{todonotes}
\usepackage{stmaryrd} 
\usepackage{hyperref}

\usepackage{setspace} 

\usepackage{tikz}
\usepackage{tkz-berge, tkz-graph}
\tikzset{vertex/.style={circle,draw,fill,inner sep=0pt,minimum size=1mm}}

\theoremstyle{plain}
\newtheorem{thm}{Theorem}
\newtheorem{lem}[thm]{Lemma}

\newtheorem{cor}[thm]{Corollary}
\newtheorem{conj}[thm]{Conjecture}

\theoremstyle{definition}
\newtheorem{definition}[thm]{Definition}
\newtheorem{exl}[thm]{Example}

\numberwithin{thm}{section}

\newcommand{\adj}{\leftrightarrow}
\newcommand{\adjeq}{\leftrightarroweq}
\DeclareMathOperator{\id}{id}
\def\Z{{\mathbb Z}}

\title{Homotopy equivalence in finite digital images}

\author[J. Haarmann]{Jason Haarmann}
\address{Dept. of Mathematics, Eastern Illinois University, Charleston IL, 61920, USA} 
\email{jhaarmann@eiu.edu}
\author[M. Murphy]{Meg P. Murphy}
\address{Dept. of Mathematics, University of North Carolina at Asheville, Asheville NC, 28804, USA}
\email{meg.page.murphy@gmail.com}
\author[C. Peters]{Casey S. Peters}
\address{Mathematics Department, University of Redlands, Redlands CA, 92373, USA}
\email{casey\_peters@redlands.edu}
\author[P. C. Staecker]{P. Christopher Staecker}
\address{Dept. of Mathematics, Fairfield University, Fairfield CT, 06824, USA}
\email{cstaecker@fairfield.edu}
\begin{document}

\begin{abstract}
For digital images, there is an established homotopy equivalence relation which parallels that of classical topology. Many classical homotopy equivalence invariants, such as the Euler characteristic and the homology groups, do not remain invariants in the digital setting. This paper develops a numerical digital homotopy invariant and begins to catalog all possible connected digital images on a small number of points, up to homotopy equivalence. 
\end{abstract}

\maketitle

\section{Introduction}
The field of digital topology has been developed over the past few decades, motivated by computer graphics, image processing, and other applications. See \cite{rosenfield73} for an early foundational work. Three main settings for studying digital objects have emerged: the Khalimsky topology on $\Z^n$ \cite{khalimsky90, slapal13},  geometric realizations of subsets of $\Z^n$ into $\mathbb R^n$ \cite{bykov99}, and graph-like adjacency structures on discrete sets \cite{boxer99,boxer11,ege14} typically called digital images.
This paper works in the final setting but de-emphasizes the traditional focus on spaces given by subsets of $\Z^n$ with various ``rectangular'' adjacency grids. 

As in classical topology, invariants are of particular interest in the study of digital spaces. Homeomorphism invariants, such as the fundamental group, homology groups, and the Euler characteristic are studied in \cite{boxer99, boxer11, ege14}.  
But homeomorphisms of digital images are very strict (for example, they must preserve the number of points), so it is natural to expand this study to homotopy equivalences, which allow for more freedom. In classical topology, the Euler characteristic and homology groups are also homotopy equivalence invariants, but a counterexample is given in \cite{ege14} showing this does not hold for digital images. 

Digital homotopy equivalence often behaves counterintuitively when compared to classical topology. For example, we will see that all cycles of fewer than five points are homotopy equivalent to a point, while any two cycles of different lengths greater than or equal to 5 will not be homotopy equivalent to one another (this result also appears in \cite{boxer05}). The aim of this paper is to develop tools and a numerical invariant to capture properties which are preserved by homotopy equivalences and to use these tools to classify digital images by their homotopy types. Unlike several other existing invariants for digital images, ours is a ``true'' digital invariant: it is not analogous to any invariant in classical topology.

Homotopy equivalence of digital images is defined in the obvious way as in \cite{boxer05}. While the equivalence relation has been mentioned in several papers, no numerical invariants seem to have been developed (other than the number of connected components, which is easily seen to be an  invariant). We note that there is an established notion of homotopy equivalence for graphs described in \cite{chen01}; however, the homotopy equivalence relation used in that theory is not equivalent to the one used in digital topology. A catalog of ``irreducible'' graphs similar to our own, but using this other homotopy notion, appears in \cite{knill12}. Our homotopy relation more closely parallels the homotopy equivalence of classical topology.

The paper is organized as follows: In Section \ref{prelims} we present the necessary background. Then in Section \ref{Lminvariantsec} we propose a ``loop-counting'' homotopy equivalence invariant for digital images and compute this number for some digital images. In Section \ref{catalog} we catalog all connected digital images on 7 and fewer points, up to homotopy equivalence, and give some partial results for images on 8 points.
In Section \ref{question} we show how our techniques answer a question posed by Boxer in \cite{boxer05}. 

This paper is the product of a summer REU project at Fairfield University supported the National Science Foundation (DMS-1358454). The authors would like to thank Laurence Boxer for many helpful corrections and suggestions. 

\section{Preliminaries}\label{prelims}

A \emph{digital image} is a set of points $X$, with some ``adjacency relation'' $\kappa$, which is symmetric and antireflexive. Typically in digital topology the set $X$ is a subset of $\Z^n$ and the relation is based on some notion of adjacency of points in the integer lattice. We will not be considering our sets as points in the integer lattice, but allow more general sets and adjacency relations. 
Digital images not based on the integer lattice can be used (see  \cite{gray71}) to model hexagonal pixel tilings of $\mathbb R^2$
or other nonstandard pixel arrangements in $\mathbb R^n$. Most of the motivation and results of this paper concern finite images, but our machinery usually applies even for infinite images. 

The standard notation for such a digital image is $(X, \kappa)$. With the adjacency relation giving us a notion of closeness, some familiar constructions from classical topology can be made in digital images. 

\subsection{Continuous functions and homotopies}
We will use the notation $x \adj_\kappa y$ when $x$ is adjacent to $y$, and $x \adjeq_\kappa y$ when $x$ is adjacent or equal to $y$. The particular adjacency relation will usually be clear from context, and in this case we will omit the subscript.
\begin{definition} \label{def:continuity}
Let $(X, \kappa_1), (Y, \kappa_2)$ be digital images. A function $f: (X, \kappa_1) \rightarrow (Y, \kappa_2)$ is $(\kappa_1, \kappa_2)$-continuous iff whenever $x \adj_{\kappa_1} y \in \kappa_1$ then $f(x) \adjeq_{\kappa_2} f(y) \in (Y, \kappa_2)$. 
\end{definition}
For simplicity of notation, we will generally not need to reference the adjacency relation specifically. Thus we typically will denote a digital image simply by $X$, and when the appropriate adjacency relations are clear we simply call a function between digital images ``continuous''. We will refer to digital images as simply ``images'', and all functions discussed will be assumed to be continuous unless otherwise noted.

An image $X$ can be viewed naturally as a simple graph, where the vertices are points of $X$ and an edge connects two vertices $x, x' \in X$ whenever $x\adj x'$. This viewpoint is helpful for our results, and we will use graph-theoretic terminology whenever convenient. Usually when describing specific examples of images we will simply draw the graph rather than describing the image abstractly.

\begin{definition}
An \emph{isomorphism} between two digital images $X$ and $Y$ is a continuous bijection $f:X\to Y$ whose inverse is also continuous. Equivalently, $f:X\to Y$ is an isomorphism of digital images when it is an isomorphism of graphs.
\end{definition}

Isomorphisms of digital images were defined by Boxer in \cite{boxer94}, where he referred to them as \emph{digital homeomorphisms}. 
It turns out that for selfmaps of finite digital images, continuity of the inverse is automatic:
\begin{lem}\label{bijhomeo}
If $X$ is a finite image and $f:X\to X$ is a continuous bijection, then $f$ is an isomorphism.
\end{lem}
\begin{proof}
It suffices to show that the inverse function $f^{-1}:X \to X$ is continuous. Equivalently, we must show that for any $x,x'\in X$, if $f(x) \adj f(x')$ then $x \adj x'$. 

Consider the cartesian product $X^2 = X\times X$, and let $E\subset X^2$ be the ``edge set'' $E = \{ (x,y) \in X^2 \mid x \adj y\}$. Let $f^2:X^2 \to X^2$ be the map $f^2(x,y) = (f(x),f(y))$. Then $f^2$ is a bijection, and continuity of $f$ means that $f^2(E) \subseteq E$. But since $f^2$ is a bijection and $X^2$ is finite this means that $f^2(E)=E$. 

Now if $f(x)\adj f(x')$ then $(f(x),f(x')) \in E$ and so $(x,x')\in E$ since $f^2$ is a bijection on $E$, and thus $x\adj x'$ as desired.
\end{proof}
The authors would like to thank Nathaniel Eldredge for suggesting the above argument. The  lemma fails to be true if $f$ is taken to be a bijection between two different images. For example if $X$ is a disconnected image of 2 points and $Y$ is a connected image of 2 points, then a bijection from $X$ to $Y$ will be continuous but its inverse will not be.

Let $[a, b]_\Z = \{z \in \Z \mid a \leq z \leq b\}$ denote the \emph{digital interval} from $a$ to $b$. The standard adjacency relation for an interval is called $2$-adjacency, indicating that each integer in the interval is adjacent to exactly the number preceding and following it. A $(2,\kappa)$-continuous function $p: [0,k]_\Z \rightarrow (X,\kappa)$ is called a \emph{path} from $p(0)$ to $p(k)$.

\begin{definition}\label{connected}
An image $X$ is \emph{connected} if for every $x, y \in X$ there exists a path $p: [0,k]_\Z \rightarrow X$ from $x$ to $y$. Given $x\in X$, the \emph{component of $x$} is the set of all $y\in X$ such that there is a path from $x$ to $y$. 
\end{definition}

Thus in the digital setting, connected and path-connected are equivalent. Homotopies and homotopy equivalence are defined in the natural way.

\begin{definition}\label{homotopy}
Let $X, Y$ be images and let $f,g: X \rightarrow Y$ be continous functions. The maps $f$ and $g$ are \emph{homotopic} if there exists $k\in \Z$ such that there is a map $H: X \times [0, k]_\Z \rightarrow Y$ with $H(x, 0) = f(x)$ and $ H(x, k) = g(x)$ for all $x\in X$, for any fixed $t\in [0, k]_Z$, the map $H(\cdot, t): X \rightarrow Y$ is continuous, and for any fixed $x\in X$ the map $H(x, \cdot): [0, k]_\Z \rightarrow Y$ is continuous. 
\end{definition}
The map $H$ is called the \emph{homotopy}, and we write $f\simeq g$ when $f$ is homotopic to $g$. As in classical topology, homotopy of maps is an equivalence relation (though the transitivity will typically require changing the length of the digital interval $[0,k]_\Z$). In the case where $f \simeq g$ by some homotopy with $k=1$, we say that $f$ is \emph{homotopic to $g$ in one step}.

We denote the identity map on $X$ as $\id_X$, which is always continuous.

\begin{definition}\label{homotopyequivalence}
The images $X$ and $Y$ are \emph{homotopy equivalent} if there exist two continuous functions $f: X\rightarrow Y$ and $g: Y\rightarrow X$ such that $f\circ g \simeq \id_Y$ and $g\circ f \simeq \id_X$.
\end{definition}

It is shown in \cite{boxer99} that homotopy equivalence is an equivalence relation. 

\begin{definition}
A finite image $X$ is \emph{reducible} when it is homotopy equivalent to an image of fewer points. Otherwise we say $X$ is \emph{irreducible}.
\end{definition}

The following is a useful characterization of reducible images.

\begin{lem}\label{nonsurj}
A finite image $X$ is reducible if and only if $\id_X$ is homotopic to a nonsurjective map.
\end{lem}

\begin{proof}
Let $f:X\rightarrow X$ be a nonsurjective map homotopic to $\id_X$, and let $Y=f(X)$ so we may consider $f$ as a map $f:X\to Y$. If $i:Y\rightarrow X$ is the inclusion mapping, then $i\circ f= \id_Y$ and $f\circ i=f\simeq \id_X$. Thus $X$ is homotopy equivalent to $Y\subsetneq X$, making $X$ reducible.

For the converse implication of the lemma, let $X$ be reducible, so $X$ is homotopy equivalent to an image of fewer points $Z$. Therefore there exists $f:X \rightarrow Z$ and $g:Z \rightarrow X$ such that $g \circ f:X \rightarrow X$ is homotopic to $\id_X$. But we have
\[ \#g( f (X)) \le \#f(X) \le \#Z < \#X, \]
(where \# denotes cardinality), and so $g \circ f$ must be nonsurjective. 
\end{proof}

In fact a stronger version of the above is also true:
\begin{lem}\label{nonsurj1step}
A finite image $X$ is reducible if and only if $\id_X$ is homotopic in one step to a nonsurjective map.
\end{lem}

\begin{proof}
If $\id_X$ is homotopic in one step to a nonsurjective map it follows immediately from Lemma $\ref{nonsurj}$ that $X$ is reducible. So we need only show that if $X$ is reducible then $\id_X$ is homotopic in one step to a nonsurjection.

If $X$ is reducible then $\id_X$ is homotopic to a nonsurjective map $f$. Let $H:X \times [0,k] \rightarrow X$ be the homotopy from $\id_X$ to $f$. Without loss of generality assume that $H(x,t)$ is surjective for all $t<k$ and not surjective for $t=k$.
Let $g:X\to X$ be the surjective map $g(x) = H(x,k-1)$, so $g$ is homotopic to $f$ in one step. Since $g$ is surjective and continuous it is a bijection, and thus an isomorphism by Lemma \ref{bijhomeo}.

Consider then the homotopy $H':X \times [0,1] \rightarrow X$ where
\[
H'(x,0)=\id_X \quad  H'(x,1)=f(g^{-1}(x)). 
\]
The second map $H'(x,1)$ is continuous because it is a composition of continuous maps, and the time-continuity of the homotopy $H'$ is satisfied because $g$ is homotopic to $f$ in one step. 
Since $H'(x,0)=\id_X$ and $H'(x,1)$ is nonsurjective (because $f$ is nonsurjective), $\id_X$ is homotopic to a nonsurjective map in one step.
\end{proof}

We can use the above to show that a homotopy equivalence of irreducible images must actually provide an isomorphism between the images.

\begin{thm}\label{homeoequiv}
Let $X$ and $Y$ be finite irreducible images. Then $X$ is homotopy equivalent to $Y$ if and only if $X$ is isomorphic to $Y$.
\end{thm}
\begin{proof}
It suffices to assume that $X$ and $Y$ are homotopy equivalent, and show that $X$ and $Y$ are homeomorphic. Let $f:X\to Y$ and $g:Y \to X$ give a homotopy equivalence, so $g\circ f \simeq \id_X$ and $f\circ g \simeq \id_Y$. It will suffice to show that $f$ and $g$ are bijections, in which case they are homeomorphisms by Lemma \ref{bijhomeo}.

Since $X$ is irreducible and $g\circ f \simeq \id_X$, the map $g\circ f$ must be surjective by Lemma \ref{nonsurj}. Since $X$ and $Y$ are finite this means that $\#Y \ge \#X$, where $\#$ denotes cardinality. By the same reasons applied to $f\circ g$ we have $\#X \ge \#Y$, and so $\#X = \#Y$ and $f$ and $g$ are bijections.
%
%
\end{proof}

By Lemma \ref{nonsurj}, reducibility of images can be judged by which selfmaps are homotopic to the identity. In the most restrictive case the identity map is not homotopic to any other map.

\begin{definition} 
We say an image $X$ is \emph{rigid} if the only map homotopic to $\id_X$ is $\id_X$. 
\end{definition}

Note that by Lemma \ref{nonsurj}, any finite rigid image is irreducible, though we will see that the converse is not true. 

\subsection{Loops}
Examining examples of small digital images reveals surprisingly few to be irreducible. In fact, most images seem reducible to a point. The simplest demonstrably irreducible images are cycles of points. 


\begin{definition}
An \emph{$m$-gon}, denoted $C_m$, is a set of $m$ distinct points $\{c_0,c_1,\ldots,c_{m-1}\}$ such that $c_i$ is adjacent to $c_{i-1}$ and $c_{i+1}$ and no other points, where subscripts are read modulo $m$. Such a set is called a ``digital simple closed curve'' in the literature.
\end{definition}

When we refer to points in an $m$-gon, unless otherwise noted we will denote them as $c_i \in C_m$, and we will always read the subscripts modulo $m$.

We will see later (Corollary \ref{Cmirr}) that when $m\ge 5$, the $m$-gon $C_m$ is irreducible but not rigid. When $2 \leq m\le 4$, the image $C_m$ is homotopy equivalent to a point, which we will see as a consequence of Theorem \ref{fourpoints}. These results are also obtained by Boxer in \cite{boxer99}.

When an $m$-gon with $m \geq 5$ appears within a larger image, other edges in the image may cause it to lose its irreducibility. We would like to investigate which loops in an image can be reduced or made equal by homotopies. For example in the image below, the ``outer" $6$-gon can be shrunk onto the $5$-gon with a homotopy. 
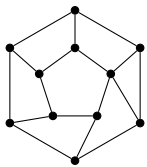
\begin{center}
\begin{tikzpicture}

\node[vertex](a) at (0,1) {};
	\node[vertex](b) at (150:1) {};
	\node[vertex](c) at (210:1) {};
	\node[vertex](d) at (270:1) {};
	\node[vertex](e) at (330:1) {};
	\node[vertex](f) at (30:1) {};
	\node[vertex](a') at (0,.5) {};
	\node[vertex](b') at (162:.5) {};
	\node[vertex](c') at (234:.5) {};
	\node[vertex](d') at (306:.5) {};
	\node[vertex](e') at (18:.5){};
	
	\draw (a) -- (b) -- (c) -- (d) -- (e) -- (f) -- (a);
	\draw (a') -- (b') -- (c') -- (d') -- (e') -- (a');
	\draw (a) -- (a'); \draw (b) -- (b'); \draw (c) -- (c'); \draw (d) -- (d'); \draw (e) -- (e');\draw (f) -- (e');

\end{tikzpicture}
\end{center}

The following definitions will formalize our efforts to capture this sort of information. We adapt definitions from Boxer in \cite{boxer99}.

\begin{definition}
In an image $X$, an $m$-loop is a map $p: C_m \rightarrow X$. If $p$ is an injection such that \ $p(c_i) \adj p(c_{i\pm 1})$ in $X$ and there are no other adjacencies between points in the image of $C_m$, then $p$ is a \emph{simple} $m$-loop. The \emph{length} of an $m$-loop is $m$. 
\end{definition} 

It will sometimes be convenient to discuss simple paths rather than loops. A \emph{simple path} in a digital image $X$ is a continuous injection $p: [0, k]_\Z \rightarrow X$ such that $p(i) \adj p(i+1)$ and there are no other adjacencies between the $p(i)$ for all $i\in [0, k-1]_\Z$.

Though we define paths and loops as maps into $X$, we will sometimes indicate them by tuples of points of $X$, so a path $p:[0,k]_\Z \to X$ would be written as $(p(0), p(1), \dots, p(k))$. We may also indicate loops similarly.

\section{A loop-counting invariant}\label{Lminvariantsec}
In this section we define an integer loop-counting invariant for images: $L_m(X)$, which counts the number of equivalence classes of $m$-loops. 

Our desired equivalence requires a slightly different notion than ordinary homotopy of maps. As in the example above, we will need to allow two loops $p:C_m\to X$ and $q:C_n\to X$ to be equivalent in some cases even when $m\neq n$, in which case it is impossible for $p\simeq q$ because the maps $p$ and $q$ have different domains. To skirt this technicality we use the concept of trivial extensions, also adapted from Boxer. Informally speaking, $\bar p$ is a trivial extension of $p$ when $\bar p$ is a path along the same points as $p$ but with some pauses for rest inserted.

\begin{definition}
[Adapted from \cite{boxer99}, Definition 4.6]
Given $C_m = \{c_0, \hdots c_{m-1}\}$, $p: C_m \rightarrow X$ a loop, we say $\bar p$ is a \emph{trivial extension} of $p$ if $\bar p: C_n \rightarrow X$, where $C_n = \{d_0, \hdots d_{n-1}\}, n\geq m$, such that there exists a sequence $t_k, k\in \{0, \hdots m-1\}$ such that $t_0 = 0$, $\bar p(d_{t_k}) = p(c_k)$, if $d_{t_k} < d_i < d_{t_{k+1}}$, then $\bar p(d_i) = p(c_{t_k})$, and $\bar p(d_i) = p(c_{m-1})$ for all $d_i > d_{t_{m-1}}$.
\end{definition}

Now we are ready to define our desired homotopy relation for simple loops, perhaps having different lengths. If $p$ and $q$ are loops in $X$, then we say $p \approx q$ if and only if $\bar p \simeq \bar q$ for some trivial extensions $\bar p, \bar q$ of $p$ and $q$ having the same length.

The appeal to trivial extensions is sometimes necessary even when $p$ and $q$ have the same length to begin with. That is, it is possible for $p \approx q$ but $p \not \simeq q$ even when $p$ and $q$ have the same length. Consider the following:
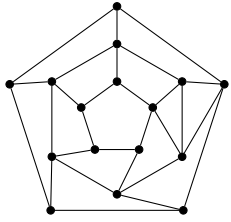
\begin{center}
\begin{tikzpicture}

\node[vertex](a) at (0,1) {};
	\node[vertex](b) at (150:1) {};
	\node[vertex](c) at (210:1) {};
	\node[vertex](d) at (270:1) {};
	\node[vertex](e) at (330:1) {};
	\node[vertex](f) at (30:1) {};
	\node[vertex](a') at (0,.5) {};
	\node[vertex](b') at (162:.5) {};
	\node[vertex](c') at (234:.5) {};
	\node[vertex](d') at (306:.5) {};
	\node[vertex](e') at (18:.5){};
	\node[vertex](a'') at (0,1.5) {};
	\node[vertex](b'') at (162:1.5) {};
	\node[vertex](c'') at (234:1.5) {};
	\node[vertex](d'') at (306:1.5) {};
	\node[vertex](e'') at (18:1.5) {};
	
	\draw (a) -- (b) -- (c) -- (d) -- (e) -- (f) -- (a);
	\draw (a') -- (b') -- (c') -- (d') -- (e') -- (a');
	\draw (a'') -- (b'') -- (c'') -- (d'') -- (e'') -- (a'');
	\draw (a'') -- (a) -- (a'); \draw (b'') -- (b) -- (b'); \draw (c'') -- (c) -- (c'); \draw (d'') -- (d) -- (d'); \draw (e'') -- (e) -- (e');\draw (e'') -- (f) -- (e');

\end{tikzpicture}
\end{center}
Let $p$ be a simple $5$-loop around the ``inner'' pentagon, and let $q$ be a simple $5$-loop around the ``outer'' pentagon (with the same orientation as $p$). Then $p\approx q$, but the equivalence between $p$ and $q$ requires a $6$-loop as an intermediate step, and so $p \not\simeq q$, since a true homotopy of $p$ to $q$ would require all intermediate steps to have length 5. 

Our proof that $\approx$ is an equivalence relation will require some elementary facts about trivial extensions. We omit the details of the proofs.
\begin{lem}\label{equivlemma}.
\begin{itemize}
\item
Let $p$ be a simple loop and let $\bar p, \hat p$ be two different trivial extensions of $p$ having the same length. Then $\bar p \simeq \hat p$.
\item
Let $p$ and $q$ be simple loops and let $\bar p \simeq \bar q$ be homotopic trivial extensions. Then for any other trivial extensions $\hat p, \hat q$ having equal lengths greater than the length of $\bar p$ and $\bar q$, we have $\hat p \simeq \hat q$. 
\end{itemize}
\end{lem}

\begin{thm}\label{equivrel}
The above relation $\approx$ is an equivalence relation.
\end{thm}

\begin{proof}
Reflexivity and symmetry are inherited from the homotopy equivalence relation. 

To prove transitivity suppose $p, q, r$ are loops in an image $X$ such that $p \approx q$ and $q \approx r$. Thus we have trivial extensions $\bar p$ and $\bar q$ with $\bar p \simeq \bar q$ and $\hat q$ and $\hat r$ with $\hat q \simeq \hat r$. By the second statement of the lemma above we can produce further extensions (with greater lengths) $\bar p', \bar q', \hat q', \hat r'$ all having the same length and we will have $\bar p' \simeq \bar q'$ and $\hat q' \simeq \hat r'$. By the first statement of the lemma abovee have $\bar q' \simeq \hat q'$ since they are two different trivial extensions of $q$ having the same length. Thus we have $\bar p' \simeq \bar q' \simeq \hat q' \simeq \hat r'$ and so $\bar p' \simeq \hat r'$, and since these are trivial extensions of $p$ and $r$ we have $p \approx r$.
\end{proof}

We say a simple $m$-loop $p$ is \emph{irreducible} if there is no $n$-loop $q \approx p$ with $n < m$. We are now ready to define our invariant:

\begin{definition}
For any positive integer $m$, let $L_m(X)$ be the number of equivalence classes of simple irreducible $m$-loops with respect to $\approx$.
\end{definition}

The number $L_m(X)$ is a homotopy equivalence invariant:

\begin{thm}\label{Lminvariant}
If $X$ and $Y$ are homotopy equivalent, then $L_m(X)=L_m(Y)$ for all $m\in \Z$. 
\end{thm}
\begin{proof}
Let $X$ and $Y$ be homotopy equivalent images, and let $f:X\rightarrow Y$ and $g:Y\rightarrow X$ be the homotopy equivalence. 

First we show that if $p:C_m\to X$ is irreducible then $f\circ p$ is also irreducible. We prove the contrapositive of this statement: assume that $f\circ p$ is reducible in $Y$. Then $f\circ p\approx q$ where $q$ has less than $m$ points. Then $g\circ q$ must have fewer than $m$ points, but $g\circ q\approx g\circ f\circ p\simeq p$ since $g\circ f \simeq \id_X$, and thus $p$ is reducible. 

Now to prove the theorem, by way of contradiction assume without loss of generality that for some $m$ we have $L_m(X)>L_m(Y)$. Then by the pigeonhole principle and the above paragraph, there exists two irreducible $m$-loops, $p$ and $q$ in $X$, such that $p\not\approx q$ but $f\circ p\approx f\circ q$. Because $g$ is continuous, $g\circ f\circ p\approx g\circ f\circ q$. Since $g\circ f\simeq \id_X$, this implies that $p\approx q$, a contradiction. Thus $L_m(X)=L_m(Y)$.
\end{proof}

The following two theorems show that, for  $m\le 4$, the values of $L_m(X)$ are very predictable.

\begin{thm}\label{L1}
For any image, $L_1(X)$ is the number of connected components of $X$.
\end{thm}

\begin{proof}
For 1-loops $p$ and $q$, the relation $p\approx q$ is equivalent to the existence of a path from $p(c_0)$ to $q(c_0)$. Thus the number of equivalence classes equals the number of connected components.
\end{proof}

\begin{thm}\label{L2L3L4}
For any image $X$, we have $L_2(X) = L_3(X) = L_4(X) = 0$.
\end{thm}


\begin{proof}
It suffices to show that any $m$-loop is reducible when $m\in \{2,3,4\}$. First we consider $m\in \{2,3\}$. Let $p:C_m \to X$ be a loop with $m\in \{2,3\}$, and let $q:C_m \to X$ be the loop given by $q(c_i) = p(c_0)$ for all $i$. Then $p \simeq q$ in one step because $p(c_i) \adjeq q(c_i) = p(c_0)$ for all $i$, but $q$ is a trivial extension of a 1-loop, so $p$ is reducible.

The case $m=4$ requires a different homotopy: Let $p:C_4 \to X$ be a 4-loop and let $q:C_4 \to X$ be given by $q(c_0) = q(c_1) = p(c_0)$ and $q(c_2) = q(c_3) = p(c_3)$. Then it is routine to check that $q$ is continuous and $p \simeq q$, but $q$ is a trivial extension of a 2-loop, and so $p$ is reducible.
\end{proof}

Throughout the paper we generally restrict our attention to connected images because all of our invariants will satisfy the following additivity on connected components:
\begin{lem}\label{lem:sumofparts}
If $X$ is finite with connected components $X_1,\dots, X_n$, then $L_m(X)=\sum\limits_{i} L_m(X_i)$.
\end{lem}

\begin{proof} 
It suffices to show that if $p$ and $q$ are simple irreducible loops whose images lie in different components (without loss of generality, say $X_1$ and $X_2$), then $p\not \approx q$.
Let $p: C_m \rightarrow X_1$ and $q: C_m \rightarrow X_2$ be simple $m$-loops. For a contradiction, suppose that $p \approx q$. Then we have $\bar p \simeq \bar q$ for some trivial extensions $\bar p, \bar q:C_n \to X$ with $n\ge m$. Let $H: C_n \times [0, k]_\Z \rightarrow X$ be the homotopy from $\bar p$ to $\bar q$. Then $H(c_0,\cdot): [0,k]_\Z \to X$ is a continuous path from $\bar p(c_0)$ to $\bar q(c_0)$, which is a contradiction since $\bar p(C_n) = p(C_m) \subseteq X_1$ and $\bar q(C_n) = q(C_m) \subseteq X_2$ and $X_1$ and $X_2$ are different components.
\end{proof}

We now develop some machinery with the ultimate goal of computing the values of $L_i(C_m)$ for all $i$, which are intuitively clear. In the case $m>4$ we have $L_1(C_m) = 1$ and $L_m(C_m) = 2$ (the latter is 2 because there are two loops ``around'' the image having opposite orientations, and these are not homotopic), and all other $L_i(C_m) = 0$. This will be proved in Theorem \ref{LmCm}. The proof is based on strong lemmas concerning certain types of paths which will play a large role for the rest of the paper. 

\begin{definition}
We say that a simple path $p:[0,k]_\Z \to X$ \emph{has no right angles} if for every $i \in \{1,\dots,k-1\}$, the two edges from $p(i-1)$ to $p(i)$ and $p(i)$ to $p(i+1)$ do not form two edges of any 4-loop (simple or otherwise) in $X$. Similarly we say that a simple loop $p:C_m \to X$ has no right angles if no two consecutive (modulo $m$) edges of $p$ form two edges of a 4-loop in $X$.
\end{definition}

\begin{lem}\label{motherlemma}
Let $p:[0,k]_\Z \to X$ be a simple path with no right angles, and let $q$ be another path homotopic to $p$ in one step with $q(1) = p(0)$. Then $q(i) = p(i-1)$ for all $i\in \{1,\dots, k\}$.
\end{lem}
\begin{proof}
Our proof is by induction on on $i$. The base case $i=1$ is our hypothesis that $q(1) = p(0)$. For the inductive step assume that $q(i-1) = p(i-2)$, and we will show that $q(i) = p(i-1)$.

Since $q$ is continuous we will have $q(i) \adjeq q(i-1) = p(i-2)$, and since $q$ is homotopic to $p$ in one step we have $q(i) \adjeq p(i)$. Since $p(i)\not\adj p(i-2)$, the point $q(i)$ cannot equal either of $p(i)$ or $p(i-2)$. Thus $q(i)$ is adjacent but not equal to both of $p(i)$ and $p(i-2)$. 

We claim in fact that $p(i-1)$ is the only point in $X$ adjacent to both of $p(i)$ and $p(i-2)$. If there were some other such point $x$, then the path $(p(i),p(i-1), p(i-2), x)$ would form a 4-loop, two of whose edges are $p(i)$ to $p(i-1)$ and $p(i-1)$ to $p(i-2)$. This would contradict our hypothesis that $p$ has no right angles. Thus $p(i-1)$ is the only point adjacent to both of $p(i)$ and $p(i-2)$, and so $q(i) = p(i-1)$ as desired.
\end{proof}

Often we will only require the following immediate corollary:
\begin{lem}\label{pathpulling}
Let $p:[0,k]_\Z \to X$ be a path with no right angles and $f:X\to X$ be homotopic to the identity in one step with $f(p(1)) = p(0)$. Then $f(p(k)) = p(k-1)$. 
\end{lem}
\begin{proof}
Let $q:[0,k]_\Z \to X$ be given by $q = f\circ p$. Then $q(1) = f(p(1)) = p(0)$ by hypothesis, and so we may apply Lemma \ref{motherlemma} to obtain $q(i)=p(i-1)$ for all $i$. In particular for $i=k$ we have $q(k) = f(p(k)) = p(k-1)$ as desired.
\end{proof}

We refer to the above lemma and its corrolary informally as a ``path-pulling'' lemma because it states that a right-angle-free path must ``pull'' along its whole length when its second point is moved to its first.

We obtain a result for loops similar to Lemma \ref{motherlemma} if we additionally assume that our loops do not pass through 3-loops.

\begin{thm}\label{looprotation}
Let $p,q:C_m \to X$ be simple $m$-loops in $X$ with no right angles and no edge of $p$ or $q$ part of a 3-loop. Then $p \simeq q$ if and only if there is some $k$ with $p(c_i)=q(c_{i+k})$.
\end{thm}
\begin{proof}
If $p(c_i) = q(c_{i+k})$ then $p\simeq q$ by the homotopy $H(x,t) = q(c_{i+t})$ for $t \in \{0, \dots, k\}$, so one implication of the theorem is clear. For the converse, it suffices to show that if $p$ is homotopic to $q$ in one step, then $p(c_i) = q(c_{i+\epsilon})$ for some $\epsilon \in \{-1,0,1\}$ and all $i$. Since $p$ is homotopic to $q$ in one step, we have $q(c_1) \in \{p(c_0), p(c_1), p(c_2)\}$. 

For the case where $q(c_1) = p(c_0)$ we will show that $q(c_i) = p(c_{i-1})$ for all $i$. Consider the paths $\bar p, \bar q: [0,m-1]_\Z \to X$ given by $\bar p(i) = p(c_i)$ and $\bar q(i) = q(c_i)$. These paths have no right angles, are homotopic in 1 step, and $\bar q(1) = \bar p(0)$. Thus by Lemma \ref{motherlemma} we have $\bar q(i) = \bar p(i-1)$ for all $i\in \{1,\dots m-1\}$, and so $q(c_i) = p(c_{i-1})$ for all such $i$. It remains to show that $q(c_0) = p(c_{m-1})$, but this is seen by applying the same argument above to paths with domain on the interval $[1,m]_\Z$. This shows that $q(c_m) = p(c_{m-1})$, which gives the desired equality because $c_m=c_0$.

For the case where $q(c_1) = p(c_2)$, consider the paths $\bar p, \bar q: [0,m-1]_\Z \to X$ given by $\bar p(i) = p(c_{2-i})$ and $\bar q(i) = q(c_{2-i})$. Then $\bar q(1) = q(c_1) = p(c_2) = \bar p(0)$. Lemma \ref{motherlemma} applies to $\bar p$ and $\bar q$, and so by the same argument as the above paragraph we have $q(c_{2-i}) = p(c_{2-i-1})$ for all $i$, which is equivalent to $q(c_i) = p(c_{i+1})$ for all $i$.

Finally we consider the case where $q(c_1) = p(c_1)$. For the sake of a contradiction let $k$ be the first integer with $q(c_k) \neq p(c_k)$. Since $q$ is homotopic to $p$ in one step we have $q(c_k) \adj p(c_k)$ and by continuity of $p$ and $q$ we have $p(c_{k-1})=q(c_{k-1})$ adjacent to both $p(c_k)$ and $q(c_k)$. Thus $p(c_{k-1}), p(c_k), q(c_k)$ form a 3-loop, which contradicts our hypothesis.
\end{proof}

Theorem \ref{looprotation} can be seen as a version of Lemma \ref{motherlemma} for loops, which requires the additional hypothesis that the edges of $p$ and $q$ are not part of any 3-loop. To see that this extra hypothesis is necessary consider the following example:
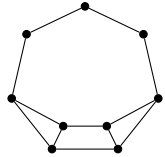
\begin{center}
\begin{tikzpicture}
	\node[vertex](a) at (0,1) {};
	\node[vertex](b) at (141:1) {};
	\node[vertex](c) at (193:1) {};
	\node[vertex](d1) at (244:1) {};
	\node[vertex](e1) at (296:1) {};
	\node[vertex](d2) at (244:.66) {};
	\node[vertex](e2) at (296:.66) {};
	\node[vertex](f) at (347:1) {};
	\node[vertex](g) at (39:1) {};
	
	\draw (a) -- (b) -- (c) -- (d1) -- (e1) -- (f) -- (g) -- (a);
	\draw (c) -- (d2) -- (e2) -- (f);
	\draw (d1) -- (d2);
	\draw (e1) -- (e2);	
\end{tikzpicture}
\end{center}
The ``inner'' and ``outer'' 7-loops are homotopic and have no right angles, but do not obey the conclusion of Theorem \ref{looprotation}.

From Theorem \ref{looprotation} we obtain immediately:
\begin{cor}\label{Cmirr}
For any $m\ge 5$, the image $C_m$ is irreducible but not rigid. 
\end{cor}
\begin{proof}
Selfmaps $f:C_m \to C_m$ can be viewed as $m$-loops in $C_m$. Since $m\ge 5$, simple $m$-loops in $C_m$ will automatically have no right angles and no edges in 3-loops, and thus we may use Theorem \ref{looprotation}. We see that the only maps homotopic to $\id_{C_m}$ have the form $f(c_i) = c_{i+k}$ for various $k$. Since these are all surjective, $C_m$ is irreducible by Lemma \ref{nonsurj}. Since these are not all equal to the identity, $C_m$ is not rigid.
\end{proof}

Experimentation with small images suggests that images of the above type are rare, so it is natural to suspect that $C_m$ for $m\ge 5$ are the only irreducible images which are not rigid. This seems to be false, however: 
\begin{exl}\label{kleingon}
Let $X$ be the following image:
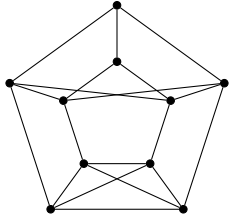
\begin{center}
\begin{tikzpicture}[scale=.75]
	\node[vertex] (a1)   at (0,1) {};
	\node[vertex] (b1)   at (162:1) {}; 
	\node[vertex] (c1) at (234:1) {};
	\node[vertex] (d1) at (306:1)   {};
	\node[vertex] (e1) at (18:1)  {};
	\node[vertex] (a2)   at (0,2) {};
	\node[vertex] (b2)   at (162:2) {}; 
	\node[vertex] (c2) at (234:2) {};
	\node[vertex] (d2) at (306:2)   {};
	\node[vertex] (e2) at (18:2)  {};
		
	\draw (a1) -- (b1) -- (c1) -- (d1) -- (e1) -- (a1);
	\draw (a2) -- (b2) -- (c2) -- (d2) -- (e2) -- (a2);
	\draw (a1) -- (a2);
	\draw (b1) -- (b2);
	\draw (c1) -- (c2);
	\draw (d1) -- (d2);
	\draw (e1) -- (e2);	
	\draw (b1) -- (e2);
	\draw (c1) -- (d2);
	\draw (d1) -- (c2);
	\draw (e1) -- (b2);
\end{tikzpicture}
\end{center}
Let $f:X\to X$ be the map which interchanges the points on the inner pentagon with their corresponding points on the outer pentagon. Then $f\simeq \id_X$ and so $X$ is not rigid. The image $X$ does, however, seem to be irreducible, though we will not prove it.
\end{exl}

Theorem \ref{looprotation} immediately shows that $C_m$ is irreducible. In fact with a bit more work we can show that many more images are irreducible:
\begin{thm}\label{34loopirr}
If $X$ is finite and has no simple 3 or 4-loops and no vertex of degree one then X is irreducible.
\end{thm}

The theorem above follows immediately from the following more technical lemma:
\begin{lem}
Let $X$ be finite and connected with no simple 3 or 4-loops, and let $f:X\to X$ be a map homotopic to the identity in one step such that $f(x)=x$ whenever $x$ has degree 1. Then $f$ is surjective.
\end{lem}

\begin{proof}
We prove the lemma by induction on the number of points in $X$. If $X$ is a single point then $f$ is automatically surjective.

For the inductive case assume for the sake of a contradiction that $f$ is not surjective, and so there are points $a,b\in X$ with $f(a)=f(b)$. Since $f$ is homotopic to the identity we have $f(x)\adjeq x$ for all $x\in X$. Thus $f(a) = f(b)$ must be adjacent or equal to both $a$ and $b$. Since our image has no 3-loops, the only points adjacent to both $a$ and $b$ are $a$ and $b$ themselves. Thus $f(a)=f(b) \in \{a,b\}$, and without loss of generality we will assume that $f(a)=f(b)=b$. 

First we show that $a$ and $b$ cannot be part of some simple $m$-loop $p$ with $m\ge 5$. Assuming such a loop $p$ did exist, since $X$ has no 3- or 4-loops we may apply Lemma \ref{looprotation} and conclude that $p$ is is homotopic only to its ``rotations'' $r_k(i) = p(i+k)$. Since $f\simeq \id_X$ we will have $f\circ p \simeq p$, and so $f\circ p$ is a rotation of $p$. But this is impossible since $f(a)=f(b)$ and $a$ and $b$ are points of $p$. Thus, below we may assume that $a$ and $b$ are not part of the same simple $m$-loop for any $m\ge 3$. (By hypothesis, our image has no simple 3- or 4-loops at all.)

Let $X'$ be the component of $a$ in $X - \{b\}$. The above paragraph means that no point of $X'$ is adjacent (in $X$) to $b$ other than $a$. If there were some other such point $c$, we would have a path from $a$ to $c$ in $X'$ (therefore disjoint from $b$), which together with the adjacencies $a\adj b \adj c$ would make a loop including $a$ and $b$, which is not possible. This in particular means that the degree of every point $x \neq a$ is the same whether we consider the degree in $X$ or the degree in $X'$. 

The previous paragraph also means that $f(x)=b$ is impossible for any $x\in X'$ with $x\neq a$, since $x\adj f(x) = b$ and $a$ is the only point of $X'$ adjacent to $b$.

Let $f':X'\to X$ be given by $f'(a)=a$ and $f'(x)=f(x)$ for all $x \neq a$. We will argue that $f'$ is in fact continuous and that $f'(X') \subseteq X'$. First we show the inclusion. Take some $x\in X'$, and let us assume for the sake of a contradiction that $f'(x)\not\in X'$. We may assume that $x\neq a$ since $f'(a)\in X'$, and so $f'(x)=f(x) \adjeq x$. 
Thus $x$ is in $X'$, while the adjacent point $f'(x)\not \in X'$. Since $X'$ is a component in $X-\{b\}$, this is possible only when $f'(x)=b$, which we have already said is impossible for any $x\neq a$. 

Now we show that $f'$ is continuous: Take two points $x,y\in X'$ with $x\adj y$, and we will show that $f'(x)\adjeq f'(y)$. If neither $x$ nor $y$ equals $a$ this is clear since $f'(x)=f(x)$ and $f'(y)=f(y)$ and $f$ is continuous. The desired adjacency is also clear when $x=y=a$. It remains to consider when one of the points is $a$. In this case we must show that if $x\adj a$, then $f'(x) \adjeq f'(a)$, which is to say $f(x) \adjeq a$. In fact it must be that $f(x)=a$, by Lemma \ref{pathpulling} applied to the path $(b,a,x)$. 

We have shown that $f':X'\to X'$ is continuous. Observe that $X'$ has no 3- or 4-loops and $f'$ fixes any points of degree 1 in $X'$, since $f'(a)=a$ and if $x\neq a$ has degree 1 in $X'$ then $x$ has degree 1 in $X$ and thus $f'(x)=f(x)=x$. Since $X'$ has fewer points than $X$ we apply induction and conclude that $f'$ is surjective. Since $X'$ is finite $f'$ is also injective.

But we have already seen that $f'(x)=a$ for any point $x \in X'$ with $x \adj a$. Thus $f$ cannot be injective, which is a contradiction, except in the special case when there are no points $x\in X'$ adjacent to $a$. Since $X'$ is the component of $X-\{b\}$ containing $a$ this special case only occurs when $a$ has degree 1 in $X$. But $f(a)=b$ and this would contradict the assumption that $f$ fixes all degree 1 points.
\end{proof}

Now we are ready to give our computation of the $L$ invariant for $m$-gons.

\begin{thm}\label{LmCm}
For $m \ge 5$ we have $L_1(C_m) = 1$ and $L_m(C_m) = 2$, and $L_i(C_m) = 0$ for all other $i$. 
\end{thm}

\begin{proof}
The fact that $L_1(C_m) = 1$ is by Theorem \ref{L1}, and $L_2(C_m) = 0$ by Theorem \ref{L2L3L4}. Also $L_i(C_m) = 0$ for all $i \not\in\{1,2,m\}$ because $C_m$ has no simple $i$-loops for such $i$. It then remains to show $L_m(C_m) = 2$.

Let $q_1,q_2: C_m \to C_m$ be simple $m$-loops given by $q_1(c_i)=c_i$ for all $i$, and $q_2(c_i) = c_{m-i}$. We will show that all simple $m$-loops in $C_m$ are homotopic to either $q_1$ or $q_2$.

Let $p:C_m \to C_m$ be any simple $m$-loop. By rotating the loop if necessary, our loop $p$ is homotopic to a loop $p'$ with $p'(c_0) = c_0$. By continuity we have $p'(c_1) \adjeq c_0$, but since $p'$ is a bijection (because $p$ is simple) and $p'(c_0) = c_0$ we have $p'(c_1) \neq c_0$ and so $p'(c_1) \in \{c_{-1}, c_1\}$. We will see below that both of these values for $p'(c_1)$ are permissible, and they determine completely the other values of $p'$.

In the first case where $p'(c_1) = c_1$, by continuity we have $p'(c_2) \adjeq c_1$ but $p'$ is a bijection and $p'(c_0)=c_0$ and $p'(c_1)=c_1$, so $p'(c_2)$ cannot equal $c_0$ or $c_1$, so $p'(c_2) = c_2$. By the same arguments it is easy to see that $p'(c_i)=c_i$ for all $i$, and thus $p' = q_1$.

In the second case where $p'(c_1)=c_{-1}$, it follows from the same reasons as above that $p'(c_i) = c_{-i}$ for all $i$. In other words, $p'(c_i) = c_{m-i}$ for all $i$, and thus $p' = q_2$.

In either case, $p$ is homotopic to one of $q_1$ or $q_2$. Thus there are at most two homotopy classes for simple $m$-loops in $C_m$, and so $L_m(C_m) \le 2$. The loops $q_1$ and $q_2$ are in fact not homotopic by Lemma \ref{looprotation} because $q_2(c_i)=c_{m-i}$ cannot be written as $q_1(c_{i+k})= c_{i+k}$ for any $k$. Thus $L_m(C_m) = 2$. 
\end{proof}

We conclude our discussion of $L_m$ with two examples. First we compute all $L_m$ numbers for a specific image that is not an $m$-gon. 
\begin{exl}\label{screwexl}
Consider the following image:
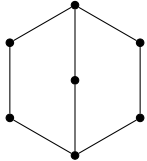
\begin{center}
\begin{tikzpicture}
	\node[vertex](a) at (0,1) {};
	\node[vertex](b) at (150:1) {};
	\node[vertex](c) at (210:1) {};
	\node[vertex](d) at (270:1) {};
	\node[vertex](e) at (330:1) {};
	\node[vertex](f) at (30:1) {};
	\node[vertex](g) at (0,0) {};
	\draw (a) -- (b) -- (c) -- (d) -- (e) -- (f) -- (a);
	\draw (a) -- (g) -- (d);
\end{tikzpicture}
\end{center}
This image of 7 points appears in our catalog of irreducible images, where it has been labeled $7_1$. By Theorem \ref{L1} we have $L_1(7_1)=1$, and $L_2(7_1) = 0$ by Theorem \ref{L2L3L4}. By Theorem \ref{looprotation} we have $L_5(7_1)=4$, since there are four simple irreducible 5-loops, and none of these are rotations of one another. Finally $L_6(7_1)=2$ because the path traversed by the outer six points is irreducible and not homotopic to its reverse. All other numbers $L_m(7_1)$ are zero since there are no simple $m$-loops for $m\not\in \{1,2,5,6\}$.
\end{exl}

Given the example above and the proof of Theorem \ref{looprotation} it would seem that the numbers $L_m(X)$ for $m>4$ should always be even, since every irreducible simple $m$-loop $p(c_i) = c_{j_i}$ will be distinct from its ``reverse'' loop $p'(c_i) = c_{m-j_i}$. The image from Example \ref{kleingon} shows that an irreducible simple $m$-loop can in fact be homotopic to its own reverse. In this sense the image is a digital analogue of the Klein bottle.

\begin{exl}
Let $X$ be the image in Example \ref{kleingon}. 
Let $p$ be the simple $5$-loop which traverses the inner pentagon. This $p$ can be pushed radially ``outwards'', and is homotopic to the simple $5$-loop which traverses the outer pentagon. Now we can push $p$ back inwards along the longer edges in the picture, and we obtain the reverse of our original $p$. Thus $p$ is homotopic to its reverse. It seems (though we will not prove it) that $L_5(X) = 1$.
\end{exl}

We conclude the section by noting that there seems to be considerable opportunity for improvement in defining loop-counting invariants for digital images. In particular the fact that $L_3(X)=L_4(X)=0$ for all $X$ means that our invariant cannot detect 3- or 4-loops, even in cases where those loops persist under homotopy equivalence. Consider for example the following space:

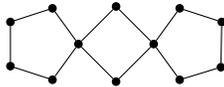
\begin{center}
\begin{tikzpicture}
	\node[vertex] (a)   at (.5,0) {};
	\node[vertex] (b)   at (72:.5) {}; 
	\node[vertex] (c) at (144:.5) {};
	\node[vertex] (d) at (216:.5)   {};
	\node[vertex] (e) at (288:.5)  {};
		
	\begin{scope}[shift={(1,0)}] 
		\node[vertex] (t) at (0,.5) {};
		\node[vertex] (s) at (0,-.5) {};
	\end{scope}
	
	\begin{scope}[shift={(2,0)}] 
		\node[vertex] (a1)   at (-.5,0) {};
		\node[vertex] (b1)   at (252:.5) {};
		\node[vertex] (c1) at (322:.5) {};
		\node[vertex] (d1) at (36:.5)   {};
		\node[vertex] (e1) at (108:.5)  {};
	\end{scope}

	\draw (a) -- (b) -- (c) -- (d) -- (e) -- (a);
	\draw (a1) -- (b1) -- (c1) -- (d1) -- (e1) -- (a1);
	\draw (t) -- (a) -- (s) -- (a1) -- (t);
\end{tikzpicture} 
\end{center}

It seems, though we will not prove it, that all images homotopy equivalent to the above will have a 4-loop. The 4-loop is itself reducible (because all 4-loops are reducible), but it seems that it cannot be reduced by a map of the whole image. It would be desirable to develop an ``ambient'' loop-counting invariant which would be nonzero for the above image.

\section{A catalog of small digital images up to homotopy equivalence}\label{catalog}
\begin{figure}
\renewcommand{\arraystretch}{7}
\begin{center}
\begin{tabular}{rccc}
$n=1$: &
\begin{minipage}{.60\textwidth}
\begin{center}
\begin{tikzpicture}
	\node[vertex] ()   at (0,0) {};
\end{tikzpicture}
\end{center}
\end{minipage}%
\\ 
$n=5$: &
\begin{minipage}{.60\textwidth}
\begin{center}
\begin{tikzpicture}
	\node[vertex] (a)   at (0,1) {};
	\node[vertex] (b)   at (162:1) {}; 
	\node[vertex] (c) at (234:1) {};
	\node[vertex] (d) at (306:1)   {};
	\node[vertex] (e) at (18:1)  {};
		
	\draw (a) -- (b) -- (c) -- (d) -- (e) -- (a);
\end{tikzpicture}
\end{center}
\end{minipage}%
\\ 
$n=6$: &
\begin{minipage}{.60\textwidth}
\begin{center}
\begin{tikzpicture}
	\node[vertex](a) at (0,1) {};
	\node[vertex](b) at (150:1) {};
	\node[vertex](c) at (210:1) {};
	\node[vertex](d) at (270:1) {};
	\node[vertex](e) at (330:1) {};
	\node[vertex](f) at (30:1) {};
	\draw (a) -- (b) -- (c) -- (d) -- (e) -- (f) -- (a);
\end{tikzpicture}
\end{center}
\end{minipage}
\\ 
$n=7$: &
\begin{minipage}{.20\textwidth}
\begin{tikzpicture}
	\node[vertex](a) at (0,1) {};
	\node[vertex](b) at (141:1) {};
	\node[vertex](c) at (193:1) {};
	\node[vertex](d) at (244:1) {};
	\node[vertex](e) at (296:1) {};
	\node[vertex](f) at (347:1) {};
	\node[vertex](g) at (39:1) {};
	\draw (a) -- (b) -- (c) -- (d) -- (e) -- (f) -- (g) -- (a);
\end{tikzpicture}
\end{minipage}%
\begin{minipage}{.20\textwidth}
\begin{tikzpicture}
	\node[vertex](a) at (0,1) {};
	\node[vertex](b) at (150:1) {};
	\node[vertex](c) at (210:1) {};
	\node[vertex](d) at (270:1) {};
	\node[vertex](e) at (330:1) {};
	\node[vertex](f) at (30:1) {};
	\node[vertex](g) at (0,0) {};
	\draw (a) -- (b) -- (c) -- (d) -- (e) -- (f) -- (a);
	\draw (a) -- (g) -- (d);
\end{tikzpicture}
\end{minipage}%
\begin{minipage}{.20\textwidth}
\begin{tikzpicture}
	\node[vertex] (a)   at (0,-1) {};
	\node[vertex] (b)   at (-162:1) {}; 
	\node[vertex] (c) at (-234:1) {};
	\node[vertex] (d) at (-306:1)   {};
	\node[vertex] (e) at (-18:1)  {};
	\node[vertex] (f) at (0,-.33) {};
	\node[vertex] (g) at (0,.33) {};
		
	\draw (a) -- (b) -- (c) -- (d) -- (e) -- (a);
	\draw (d) -- (g) -- (c);
	\draw (g) -- (f) -- (a);
\end{tikzpicture}
\end{minipage}%
\end{tabular}
\caption{Catalog of all connected digital images up to homotopy equivalence on $n$ points for $n\le 7$\label{catalogfig}}
\end{center}
\end{figure}
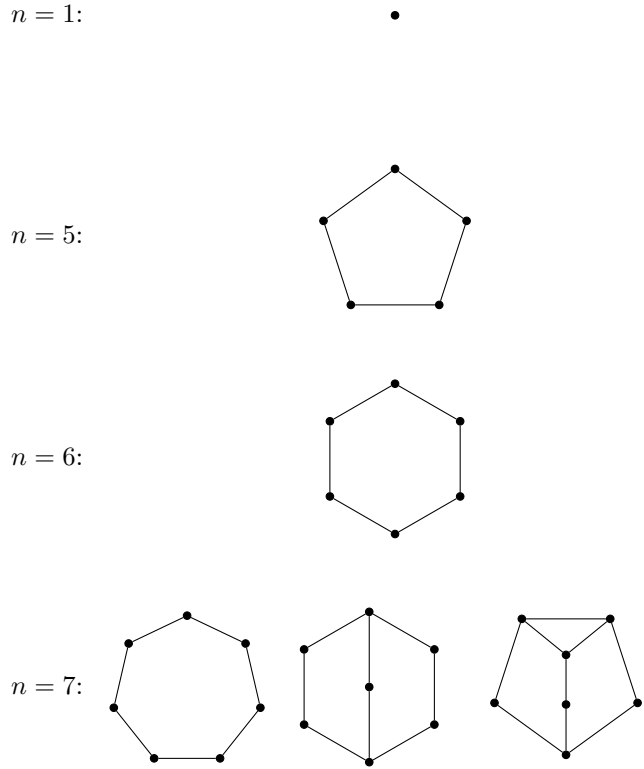

In this section, we produce an exhaustive list of all digital images of seven points or fewer up to homotopy equivalence. 
Our results are shown in Figure \ref{catalogfig}. We show that all images in the figure are of different homotopy types, and that no other images can appear in such a chart.

We begin with some easily checked criteria which will show that an image is reducible.
\begin{lem}\label{simplexreduction} 
If there exists a complete subgraph $S \subseteq X$ so that all vertices of $X$ are adjacent to points of $S$, then $X$ is homotopy equivalent to a point.
\end{lem} 

\begin{proof}
Let $f:X \rightarrow \{ s \}$ for some $s \in S$ be a constant map and $g:\{ s \} \rightarrow X$ be the inclusion map. Then $f \circ g: \{ s \} \rightarrow \{ s \}$ is the identity map and $g \circ f: X \rightarrow X$ is the constant map to $s$. Consider the homotopy $H: X \times [0,2]_\Z \rightarrow X$ defined by $H(x,0)=x$, the map $H(x,1)$ takes each point to an adjacent point in $S$, which must exist by hypothesis, and $H(x,2)=s$. This shows that $g \circ f \simeq \id_{X}$ and therefore $X$ is homotopy equivalent to $\{s\}$.
\end{proof} 

For $x\in X$, define $N(x)\subseteq X$ to be the set of points $x$ and all points adjacent to $x$. In other words $N(x)= \{ y \mid y \adjeq x \}$.

\begin{lem}\label{pointreduction}
If there exists $x,y \in X$ so that $N(x) \subseteq N(y)$, then $X$ is reducible. In particular $X$ is homotopy equivalent to $X-\{ x \}$.
\end{lem}

\begin{proof} 
Let $f:X \rightarrow X- \{x \}$ map $x$ to $y$ and all other points to themselves, and let $g:X - \{ x \} \rightarrow X$ be the inclusion map. Then $f \circ g: X-\{ x \} \rightarrow X-\{ x \}$ is the identity map and $g \circ f: X \rightarrow X$ maps $x$ to $y$ and all other points to themselves. The homotopy $H:X \times [0,1] \rightarrow X$ defined $H(y,0)=y$ and $H(y,1)=g \circ f(y)$ shows that $g \circ f \simeq \id_{X}$. Thus $X$ is homotopy equivalent to $X-\{ x \}$.
\end{proof}

We obtain several immediate corollaries to the above:
\begin{cor}\label{simplexneighborhood}
If there is any $x\in X$ such that $N(x)$ is a nontrivial complete subgraph of $X$, then $X$ is homotopy equivalent to $X - \{x\}$.
\end{cor}
\begin{proof}
If $N(x)$ is a complete subgraph then automatically we have $N(x) \subseteq N(y)$ for any $y \adj x$ and so $X$ is reducible to $X-\{x\}$ by Lemma \ref{pointreduction}.
\end{proof}

If $x\in X$ is a vertex of degree 1, then $N(x) = \{x,y\}$ for some $y\in X$, and so $N(x)$ is a complete subgraph. Thus we have:
\begin{cor}\label{degreeone}
If $x\in X$ is a vertex of degree 1, then $X$ is homotopy equivalent to $X-\{x\}$.
\end{cor}

Since nontrivial trees always include vertices of degree 1, and the above construction can remove them one at a time, we obtain:

\begin{cor}
If $X$ is a tree, then $X$ is homotopy equivalent to a point.
\end{cor}


Our next theorem is a generalization of Lemma \ref{pointreduction}, which collapses one point into another adjacent point. In the following, we see that an entire path can be collapsed into another path if the proper adjacencies are present.

\begin{thm}\label{pathreduction}
If there exist injective paths $p,q: [0,k]_\Z \to X$, such that $p(i) \adjeq q(i)$  and  $N(p(0)) \subseteq N(q(0))\cup \{p(1)\}$ and $N(p(k)) \subseteq N(q(k)) \cup \{p(k-1)\}$ and  $N(p(i)) \subset N(q(i)) \cup \{ p(i-1), p(i+1) \}$ for all $i\in\{1,2,\dots, k-1\}$, then $X$ is homotopy equivalent to $X - (p([0,k])-q([0,k]))$.
\end{thm}

\begin{proof} 
Let $W = p([0,k]_\Z)$ and $Z = q([0,k]_\Z)$. 
Let $f:X \rightarrow X-(W-Z)$ be the inclusion map (or the identity if $W-Z$ is empty). Let $g:X - (W-Z) \rightarrow X$ map $p(i)$ to $q(i)$ for all $i$ and all other points to themselves. Then $f \circ g: X-(W-Z) \rightarrow X-(W-Z)$ is the identity map and $g \circ f: X \rightarrow X$ maps $p(i)$ to $q(i)$ for all $i$ and all other points to themselves. Let $H:X \times [0,1] \rightarrow X$ be defined by $H(x,0)=x$ and $H(x,1)=f \circ g(x)$. Continuity of $H$ follows by the assumptions of the theorem, and so $f \circ g \simeq \id_{X}$. Thus $X$ is homotopy equivalent to $X-(W-Z)$.
\end{proof}

We have one final reduction theorem which we will use in developing our catalog of images:

\begin{thm}\label{chordal}
If $X$ has no simple $m$-loop for any $m\ge 4$, then $X$ is homotopy equivalent to a point.
\end{thm}

\begin{proof}
Our proof is by induction on $n$, the number of vertices of $X$. If $n=1$ then there is nothing to show. 

Viewing $X$ as a graph, our condition that there are no simple $m$-loops of degree $4$ means that $X$ is chordal \cite{diestel10}. A fundamental theorem in the study of chordal graphs is that a graph is chordal if and only if it admits a perfect elimination ordering \cite{dirac61, gross65}. A perfect elimination ordering is an ordering of the vertices $X = \{x_1, \dots, x_n\}$ such that for any $i$, the vertex $x_i$ together with any of its neighbors $x_j$ with $j>i$ form a complete subgraph of $X$.

Assuming our image $X$ has a perfect elimination ordering in this way, consider in particular the first point $x_1$. The condition above means that $x_1$, together with all its neighbors, forms a complete subgraph. Thus $N(x_1)$ is a complete subgraph and so $X$ is homotopy equivalent to $X-\{x_1\}$ by Corollary \ref{simplexneighborhood}. By induction $X-\{x_1\}$ is homotopy equivalent to a point, and thus by transitivity $X$ is reducible to a point.
\end{proof}

The statement of the above theorem can be strengthened in an obvious way which seems to be true; however, we have been unable to prove it, so we state it as a conjecture:

\begin{conj}
If $X$ has no simple $m$-loop for any $m\ge 5$, then $X$ is homotopy equivalent to a point.
\end{conj}

We are now ready to fully characterize irreducible images on 7 or fewer points. The argument is simplest for images of 4 points or fewer.

\begin{thm}\label{fourpoints}
Let $X$ be a connected image of $4$ points or fewer. Then $X$ is homotopy equivalent to a point.
\end{thm}

\begin{proof}
It suffices to show that all connected images of 2, 3, or 4 points are reducible. If $X$ has no simple 4-loops, then $X$ is reducible to a point by Lemma \ref{chordal}. In the case when $X$ has a simple 4-loop, since $X$ has at most 4 points we have $X=C_4$. But each edge in $C_4$ forms a complete subgraph of 2 vertices which satisfies Lemma \ref{simplexreduction}, and so $X$ is homotopy equivalent to a point.
\end{proof}


\begin{thm}\label{fivepoints}
Let $X$ be a connected image of $5$ points. Then $X$ is homotopy equivalent to a point or to $C_5$.
\end{thm}

\begin{proof}
We will show that all connected images of $5$ points are reducible except for $C_5$. The proof is essentially the same as for Theorem \ref{fourpoints}. By Lemma \ref{chordal}, if $X$ has no simple 4-, or 5-loops then $X$ is reducible. If $X$ has a simple 5-loop, then $X = C_5$ and there is nothing to show.

It remains to consider the case where $X$ has a simple 4-loop, say of points $c_0,c_1,c_2,c_3 \in X$. Let $z$ be the point of $X$ which is not part of the 4-loop, and without loss of generality assume that $z$ is adjacent to $c_0\in X$ (since $X$ is connected, $z$ must be adjacent to something). Then $\{ c_0,c_1 \}$ is a complete subgraph of 2 points that satisfies Lemma \ref{simplexreduction} so $X$ is homotopy equivalent to a point.
\end{proof}

It should be possible to prove the next few theorems in the style of our proofs of Theorems \ref{fourpoints} and \ref{fivepoints}, but treating all possible configurations for an image of 6 or more points becomes cumbersome. Instead we use a computer search to narrow down the possible candidates for $X$. (Similar computer searches could also be used to immediately prove Theorems \ref{fourpoints} and \ref{fivepoints}.) This is done easily with standard computer algebra systems to remove from consideration any images which are reducible by Lemmas \ref{simplexreduction} and \ref{pointreduction}. We have used the open source software Sage, and our source code is available for testing at the last author's website.\footnote{\url{http://faculty.fairfield.edu/cstaecker}}


\begin{thm}\label{sixpoints}
Let $X$ be an irreducible image of $6$ points. Then $X$ is homotopy equivalent to $C_6$.
\end{thm}
\begin{proof}
We begin with all connected graphs on 6 vertices (up to isomorphism, there are 112 of them). The computer search reveals that all of these are reducible by Lemmas \ref{simplexreduction} and \ref{pointreduction} except for the following two images:
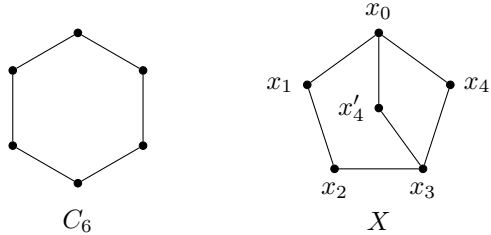
\begin{center}
\begin{tikzpicture}
	\node[vertex](a) at (0,1) {};
	\node[vertex](b) at (150:1) {};
	\node[vertex](c) at (210:1) {};
	\node[vertex](d) at (270:1) {};
	\node[vertex](e) at (330:1) {};
	\node[vertex](f) at (30:1) {};
	\draw (a) -- (b) -- (c) -- (d) -- (e) -- (f) -- (a);
	\node at (270:1.5){$C_6$};

\begin{scope}[shift={(4,0)}]

	\node[vertex] (y0)   at (0,1)  [label=above :$x_0$] {};
	\node[vertex] (y1)   at (162:1) [label=left :$x_1$] {}; 
	\node[vertex] (y2) at (234:1) [label=below :$x_2$] {};
	\node[vertex] (y3) at (306:1)  [label=below :$x_3$] {};
	\node[vertex] (y4) at (18:1) [label=right :$x_4$] {};
	\node[vertex] (y4p) at (0,0) [label=left :$x_4'$] {};
		
	\draw (y0) -- (y1) -- (y2) -- (y3) -- (y4) -- (y0);
	\draw (y3) -- (y4p) -- (y0);
	\node at (270:1.5){$X$};	
\end{scope}

\end{tikzpicture}
\end{center}
The image above labeled $X$ is reducible. Consider the homotopy given by $H(x,0) = x$ for all $x$, and $H(x_i,1) = x_{i+1}$ (subscripts read modulo $5$) and $H(x'_4,1) = x_0$. This gives a homotopy of $\id_X$ to a nonsurjective map, and so $X$ is reducible by Lemma \ref{nonsurj}.
\end{proof}

The same computer search will help to classify images of 7 points, but we have more special cases to check by hand.

\begin{thm}\label{sevenpoints}
Up to homotopy equivalence, there are 3 irreducible images of 7 points.
\end{thm}
\begin{proof}
We begin with the same computer search used in the proof of Theorem \ref{sixpoints}, this time beginning with all connected graphs of 7 vertices, of which there are 853 up to isomorphism. Eliminating those which can be reduced by Lemmas \ref{simplexreduction} and \ref{pointreduction}, we obtain 15 graphs which we must distinguish by hand. Of these 15, all are reducible except for 3 of them. Details of the reductions for the 12 reducible graphs have been included as Theorem \ref{appendixthm} in the Appendix.

What remains are the following three images of 7 points which do not seem to be reducible. 
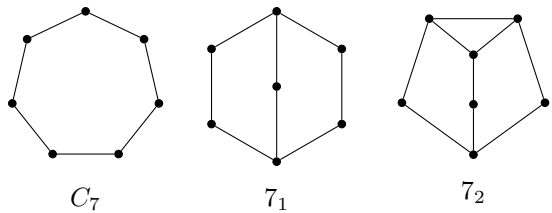
\begin{center}
\begin{minipage}{.20\textwidth}
\begin{tikzpicture}
	\node[vertex](a) at (0,1) {};
	\node[vertex](b) at (141:1) {};
	\node[vertex](c) at (193:1) {};
	\node[vertex](d) at (244:1) {};
	\node[vertex](e) at (296:1) {};
	\node[vertex](f) at (347:1) {};
	\node[vertex](g) at (39:1) {};
	\draw (a) -- (b) -- (c) -- (d) -- (e) -- (f) -- (g) -- (a);
	\node at (270:1.5){$C_7$};
\end{tikzpicture}
\end{minipage}
\begin{minipage}{.20\textwidth}
\begin{tikzpicture}
	\node[vertex](a) at (0,1) {};
	\node[vertex](b) at (150:1) {};
	\node[vertex](c) at (210:1) {};
	\node[vertex](d) at (270:1) {};
	\node[vertex](e) at (330:1) {};
	\node[vertex](f) at (30:1) {};
	\node[vertex](g) at (0,0) {};
	\draw (a) -- (b) -- (c) -- (d) -- (e) -- (f) -- (a);
	\draw (a) -- (g) -- (d);
	\node at (270:1.5){$7_1$};
\end{tikzpicture}
\end{minipage}%
\begin{minipage}{.20\textwidth}
\begin{tikzpicture}
	\node[vertex] (a)   at (0,-1) {};
	\node[vertex] (b)   at (-162:1) {}; 
	\node[vertex] (c) at (-234:1) {};
	\node[vertex] (d) at (-306:1)   {};
	\node[vertex] (e) at (-18:1)  {};
	\node[vertex] (f) at (0,-.33) {};
	\node[vertex] (g) at (0,.33) {};
		
	\draw (a) -- (b) -- (c) -- (d) -- (e) -- (a);
	\draw (d) -- (g) -- (c);
	\draw (g) -- (f) -- (a);
	\node at (0,-1.5){$7_2$};
\end{tikzpicture}
\end{minipage}%
\\ 
\end{center}
In fact each of these images are indeed irreducible: $C_7$ and $7_1$ by Lemma \ref{34loopirr} and $7_2$ will be shown to be rigid in Example \ref{mantisexl}. They are not isomorphic, and thus are not homotopy equivalent to one another by Lemma \ref{homeoequiv}.
\end{proof}

We conclude the paper by presenting some irreducible images of 8 points, though we do not claim this is the complete list. Our partial list consists of $C_8$, together with four rigid images. 
We begin with a criterion which can be used to show that an image is rigid.

\begin{definition}
A \emph{lasso} in $X$ is a simple loop $p:C_m\to X$ and a simple path $r:[0,k]_\Z \to X$ with $r(k) = p(c_0)$ such that $k>0$ and $m\ge 5$ and neither of $p(c_1)$ or $p(c_{m-1})$ are adjacent to $r(k-1)$. 

We will say such a lasso \emph{has no right angles} when each of $p$ and $r$ have no right angles, and when no right angle is formed where $r$ meets $p$. That is, the final edge of $r$, together with either of the edges in $p$ meeting $p(c_0)$, do not form two edges of any 4-loop in $X$.
\end{definition} 

Note that above we require $p$ and $r$ to be simple in themselves and at the the point $r(k)$, but points of $r$ and $p$ may be adjacent or even equal to one another away from $r(k)$.

\begin{thm}\label{lassothm}
Let $X$ be an image in which, for any two adjacent points $x \adj x'\in X$, there is a lasso with no right angles having path $r:[0,k]_\Z \to X$ with $r(1)=x$ and $r(0)=x'$. Then $X$ is rigid.
\end{thm}
\begin{proof}
It suffices to show that the only map homotopic to $\id_X$ in one step is $\id_X$.
Let $f:X\to X$ be homotopic to the identity in one step, and choose some $x\in X$. For the sake of a contradiction assume that $f(x)\neq x$.

Since $x \adj f(x)$, there is a lasso given by $p:C_m\to X$ and $r:[0,k]_Z \to X$ with $r(1) = x$ and $r(0) = f(x)$. By Lemma \ref{pathpulling} applied to $r$ we have $f(r(k)) = r(k-1)$, and so $f(p(c_0)) = r(k-1)$. Let $q:[0,2]_\Z \to X$ be the path $(r(k-1), r(k)=p(c_0), p(c_1))$. Since our lasso has no right angles, $q$ has no right angles and thus Lemma \ref{pathpulling} again gives $f(p(c_1)) = p(c_0)$. Similarly applying Lemma \ref{pathpulling} to the two-step path from $r(k-1)$ to $p(c_{m-1})$ gives $f(p(c_{m-1})) = p(c_0)$. 

Now choose some $y= p(c_j)$ with $1 < j < m-1$ (here we use the fact that $m\ge 5$), and consider the arc of $p$ from $p(c_0)$ to $y$ through $p(c_1)$. Again we apply Lemma \ref{pathpulling} and obtain $f(y) = p(c_{j-1})$. But we can consider the other arc of $p$ from $p(c_0)$ to $y$ through $p(c_{m-1})$ and this time Lemma \ref{pathpulling} gives $f(y) = p(c_{j+1})$, which is a contradiction. 
\end{proof}

\begin{exl}\label{mantisexl}
We will show that $7_2$, the image of 7 points from Theorem \ref{sevenpoints}, is rigid. It will be convenient for this argument to draw and label $7_2$ as follows:

\begin{center}
\begin{tikzpicture}[scale=0.75, every node/.style={transform shape}]
\GraphInit[vstyle=Normal]
%
%
\Vertex[L=\hbox{$0$},x=1.7cm,y=1cm]{v0}
\Vertex[L=\hbox{$1$},x=-1.7cm,y=1cm]{v1}
\Vertex[L=\hbox{$2$},x=0cm,y=-2cm]{v2}
\Vertex[L=\hbox{$3$},x=.85cm,y=.5cm]{v3}
\Vertex[L=\hbox{$4$},x=-.85cm,y=.5cm]{v4}
\Vertex[L=\hbox{$5$},x=0cm,y=-1cm]{v5}
\Vertex[L=\hbox{$6$},x=0.0cm,y=0cm]{v6}
\Edge[](v0)(v1)
\Edge[](v0)(v2)
\Edge[](v1)(v4)
\Edge[](v0)(v3)
\Edge[](v3)(v6)
\Edge[](v4)(v6)
\Edge[](v5)(v6)
\Edge[](v1)(v2)
\Edge[](v5)(v2)
\end{tikzpicture}
\end{center}

We need only verify that the condition of Theorem \ref{lassothm} holds. By symmetry, we need only consider cases where the point $x$ is one of 1, 4, or 6.

For $x=1$ by symmetry we consider only $x'=4$ and $x'=0$. For $x'=4$ we have the lasso given by the path $(1,4,6)$ and loop $(6,3,0,2,5,6)$. For $x'=0$ we have the lasso given by path $(0,1,4,6)$ and the same loop.

For $x=4$ we need to consider $x'=1$ and $x'=6$. For $x'=6$ we use path $(4,6)$ and loop $(6,3,0,2,5,6)$. For $x'=1$ we use path $(4,1,2)$ and the loop $(2,5,6,3,0,2)$.

For $x=6$ we need to consider only $x'=4$. In this case we use path $(6,4,1,2)$ and loop $(2,5,6,3,0,2)$.

Having checked all pairs $x,x' \in X$, we conclude that $7_2$ is rigid.
\end{exl}

As shown above, for many images it is routine to verify rigidity using Theorem \ref{lassothm}. For example the same arguments will work if any additional points were inserted into $7_2$ along the radial paths from 0, 1, and 2 to 6. Also, many similar images constructed with a 4-loop or any $m$-loop around the outside (instead of the triangle in $7_2$) will be rigid. Similar arguments will show that any wedge sum of loops or union of loops joined along some edges will be rigid, provided that the loops are big enough.

In particular we can construct several examples of rigid images on 8 points.
\begin{thm}\label{eightpoints}
Up to homotopy equivalence, there are at least 5 irreducible images of 8 points.
\end{thm}
\begin{proof}
We will show that the following 5 images are irreducible and not homotopy equivalent:
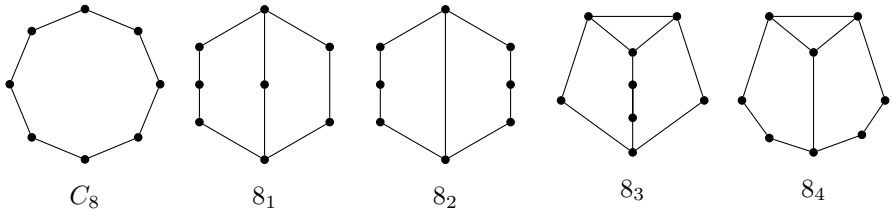
\begin{center}
\begin{minipage}{.19\textwidth}
\begin{tikzpicture}
	\node[vertex](a) at (0,1) {};
	\node[vertex](b) at (135:1) {};
	\node[vertex](c) at (180:1) {};
	\node[vertex](d) at (225:1) {};
	\node[vertex](e) at (270:1) {};
	\node[vertex](f) at (315:1) {};
	\node[vertex](g) at (1,0) {};
	\node[vertex](h) at (45:1) {};
	\draw (a) -- (b) -- (c) -- (d) -- (e) -- (f) -- (g) -- (h) -- (a);
	\node at (270:1.5){$C_8$};
\end{tikzpicture}
\end{minipage}
\begin{minipage}{.19\textwidth}
\begin{tikzpicture}
	\node[vertex](a) at (0,1) {};
	\node[vertex](b) at (150:1) {};
	\node[vertex](b1) at (-.86602,0) {};  
	\node[vertex](c) at (210:1) {};
	\node[vertex](d) at (270:1) {};
	\node[vertex](e) at (330:1) {};
	\node[vertex](f) at (30:1) {};
	\node[vertex](g) at (0,0) {};
	\draw (a) -- (b) -- (b1) -- (c) -- (d) -- (e) -- (f) -- (a);
	\draw (a) -- (g) -- (d);
	\node at (270:1.5){$8_1$};
\end{tikzpicture}
\end{minipage}%
\begin{minipage}{.19\textwidth}
\begin{tikzpicture}
	\node[vertex](a) at (0,1) {};
	\node[vertex](b) at (150:1) {};
	\node[vertex](b1) at (-.86602,0) {};  
	\node[vertex](c) at (210:1) {};
	\node[vertex](d) at (270:1) {};
	\node[vertex](e) at (330:1) {};
	\node[vertex](h) at (.86602,0) {};
	\node[vertex](f) at (30:1) {};
	\draw (a) -- (b) -- (b1) -- (c) -- (d) -- (e) -- (h) -- (f) -- (a);
	\draw (a) -- (d);
	\node at (270:1.5){$8_2$};
\end{tikzpicture}
\end{minipage}%
\begin{minipage}{.19\textwidth}
\begin{tikzpicture}
	\node[vertex] (a)   at (0,-1) {};
	\node[vertex] (b)   at (-162:1) {}; 
	\node[vertex] (c) at (-234:1) {};
	\node[vertex] (d) at (-306:1)   {};
	\node[vertex] (e) at (-18:1)  {};
	\node[vertex] (f) at (0,-.54) {};
	\node[vertex] (f1) at (0, -.1) {};
	\node[vertex] (g) at (0,.33) {};
		
	\draw (a) -- (b) -- (c) -- (d) -- (e) -- (a);
	\draw (d) -- (g) -- (c);
	\draw (g) -- (f) -- (f1) -- (a);
	\node at (0,-1.5){$8_3$};	
\end{tikzpicture}
\end{minipage}%
\begin{minipage}{.19\textwidth}
\begin{tikzpicture}
	\node[vertex] (a)   at (0,-1) {};
	\node[vertex] (b)   at (-162:1) {}; 
	\node[vertex] (b1) at (-126:1) {};
	\node[vertex] (c) at (-234:1) {};
	\node[vertex] (d) at (-306:1)   {};
	\node[vertex] (e) at (-18:1)  {};
	\node[vertex] (e1) at (-50:1) {};
	\node[vertex] (g) at (0,.33) {};
		
	\draw (a) -- (b1) -- (b) -- (c) -- (d) -- (e) -- (e1) -- (a);
	\draw (d) -- (g) -- (c);
	\draw (g) -- (a);
	\node at (0,-1.5){$8_4$};
\end{tikzpicture}
\end{minipage}%
\end{center}

$C_8$ is irreducible by Corollary \ref{Cmirr}. Each of the other images above is rigid (and thus irreducible), which can be demonstrated using Theorem \ref{lassothm} as in Example \ref{mantisexl}. We omit the details.

It remains only to show that these images are not homotopy equivalent. Since they are irreducible, by Theorem \ref{homeoequiv} it suffices to show that they are not isomorphic, and this is easily verified. For example $8_1$ has exactly two points of degree 3 which are adjacent to a common point. None of the other graphs have this configuration, so $8_1$ is not isomorphic to the others. Again we omit the details for the others.
\end{proof}

We will not compute exactly the number of irreducible images of 8 points up to homotopy equivalence, but we believe this number to be  greater than 5. A computer search similar to that used in Theorems \ref{sixpoints} and \ref{sevenpoints} shows that the number of irreducible images of 8 points up to homotopy equivalence is less than or equal to 106. The same search for images of 9 points shows the number to be less than or equal to 2132.

\section{Pointed homotopy equivalence and a question of Boxer}\label{question}
The image $X$ in the proof of Theorem \ref{sixpoints} provides an answer to a question of Boxer in \cite{boxer05}. For a pointed image $(X,x_0)$, and pointed maps $f,g:X \to X$ (with $f(x_0) = x_0$ and $g(x_0) = x_0$), a \emph{pointed homotopy} from $f$ to $g$ is a homotopy $H$ from $f$ to $g$ satisfying $H(x_0, t) = x_0$ for all $t$. Pointed images $(X,x_0)$ and $(Y,y_0)$ are \emph{pointed homotopy equivalent} if there is a homotopy equivalence $f:X\to Y$ and $g:Y \to X$ with $f\circ g \simeq \id_Y$ and $g\circ f \simeq \id_X$ by pointed homotopies. 

Boxer asks: if $X$ and $Y$ are homotopy equivalent, must they be pointed homotopy equivalent for any choice of base points $x_0$ and $y_0$? The question arises naturally when studying the digital fundamental group, where the natural category of discourse is pointed images with pointed homotopies. 

The answer to Boxer's question is ``no''. Consider the image called $X$ in the proof of Theorem \ref{sixpoints}:
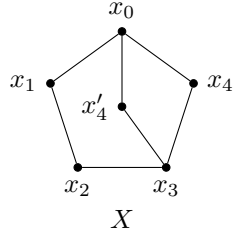
\begin{center}
\begin{tikzpicture}

	\node[vertex] (y0)   at (0,1)  [label=above :$x_0$] {};
	\node[vertex] (y1)   at (162:1) [label=left :$x_1$] {}; 
	\node[vertex] (y2) at (234:1) [label=below :$x_2$] {};
	\node[vertex] (y3) at (306:1)  [label=below :$x_3$] {};
	\node[vertex] (y4) at (18:1) [label=right :$x_4$] {};
	\node[vertex] (y4p) at (0,0) [label=left :$x_4'$] {};
		
	\draw (y0) -- (y1) -- (y2) -- (y3) -- (y4) -- (y0);
	\draw (y3) -- (y4p) -- (y0);
	\node at (270:1.5){$X$};	
\end{tikzpicture}
\end{center}
In the proof of Theorem \ref{sixpoints} we show that $X$ is homotopy equivalent to $C_5$. We will show that this equivalence cannot be realized by pointed homotopies. 

We will view $X$ as the pointed image $(X,x_0)$. It will suffice to show that no function is pointed homotopic to $\id_X$ in one step other than $\id_X$ itself. Let $f$ be a map with $f\simeq \id_X$ by a pointed homotopy in one step, and we will show that $f(x)=x$ for all $x\in X$ by path-pulling arguments using Lemma \ref{pathpulling}. 

The right angles in $X$ at $x_4$ and $x'_4$ mean that we cannot apply Lemma \ref{pathpulling} indiscriminately, but a slight variation of the lemma will hold even for certain paths through $x_4$ and $x'_4$. For example if we consider the path $(x_2,x_3,x_4,x_0)$ and assume that $f(x_3) = x_2$, we can conclude by the proof of Lemma \ref{pathpulling} that $f(x_0) \in \{x_4,x'_4\}$.

Now we show that $f$ must be the identity. Since $f\simeq \id_X$ by a pointed homotopy with base point $x_0$ we will have $f(x_0) = x_0$. We will show that $f$ is the identity on each of the other points individually, starting with $x_3$. 

Now consider $x_3$: Since $f\simeq \id_X$ in one step, $f(x_3) \in \{x_2,x_3,x_4, x'_4\}$. If $f(x_3) = x_4$ then Lemma \ref{pathpulling} applied to the path $(x_4,x_3,x_2,x_1,x_0)$ gives $f(x_0) = x_1$ which is a contradiction since $f(x_0)=x_0$. Similarly $f(x_3) \neq x'_4$. If $f(x_3) = x_2$ then our variation of Lemma \ref{pathpulling} applied to $(x_2, x_3, x_4, x_0)$ gives $f(x_0) \in \{x_4, x'_4\}$ which is again a contradiction. Thus we conclude that $f(x_3) = x_3$. 

Now consider $x_4$: Since $f$ is homotopic to the identity in one step, $f(x_4) \in \{x_0,x_4,x_3\}$. As above, if $f(x_4) = x_3$ then $f(x_0) \in \{x_4, x'_4\}$ which a contradiction. Similarly if $f(x_4) = x_0$ then $f(x_3) \in \{x_4, x'_4\}$ which contradicts the fact above that $f(x_3) = x_3$. Thus we have $f(x_4) = x_4$. For exactly the same reasons we have $f(x'_4)=x'_4$. 

It remains to show that $f(x_1) = x_1$ and $f(x_2) = x_2$. For $x_1$ we have $f(x_1) \in \{x_0, x_1, x_2\}$. 
If $f(x_1) = x_2$ then Lemma \ref{pathpulling} applied to $(x_2, x_1, x_0)$ gives $f(x_0) = x_1$ which is a contradiction. If $f(x_1) = x_0$ then Lemma \ref{pathpulling} applied to the path $(x_0, x_1, x_2, x_3)$ gives $f(x_3) = x_2$. But we have already shown that $f(x_3) = x_3$ and so we conclude that $f(x_1) = x_1$. Similar arguments show that $f(x_2) =x_2$, and thus $f = \id_X$. 

We have shown there are no maps pointed homotopic to $\id_X$ other than $\id_X$ (one could say that $X$ is ``pointed rigid''), and so the pointed image $(X,x_0)$ is not pointed homotopy equivalent to $(C_5,c)$ for any $c\in C_5$, even though $X$ is homotopy equivalent to $C_5$ when we do not require pointed homotopies. Similar arguments will show that $(X,x)$ is not pointed homotopy equivalent to $C_5$ regardless of which point $x\in X$ is chosen as the base point.

\section{Appendix: Details for the proof of Theorem \ref{sevenpoints}}
As described in the proof of Theorem \ref{sevenpoints}, a computer search of all connected graphs on 7 vertices reveals 15 images which cannot be reduced by Lemmas \ref{simplexreduction} or \ref{pointreduction}. These 15 images, $X_1, \dots, X_{15}$ are displayed in Figure \ref{sevenspacesfig}. In this section we will show that all of these images are reducible except for the first three, which are respectively $C_7$, $7_1$, and $7_2$ from the proof of Theorem \ref{sevenpoints}.

\def\sevenscale{.65}

\begin{figure}
\input{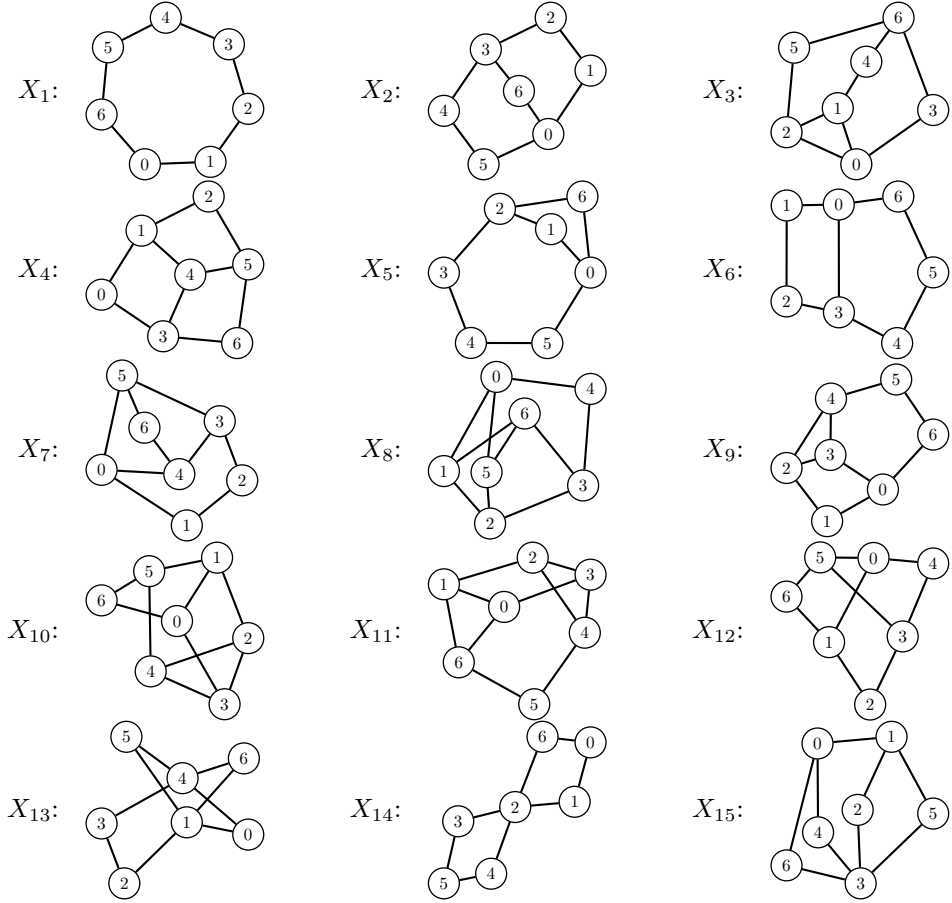}
\begin{tabular}{rlrlrl}
$X_{1}$: &
\begin{minipage}{.25\textwidth}
\sevenspacea
\end{minipage}
&
$X_{2}$: &
\begin{minipage}{.25\textwidth}
\sevenspaceb
\end{minipage}
&
$X_{3}$: &
\begin{minipage}{.25\textwidth}
\sevenspacec
\end{minipage}
\\
$X_{4}$: &
\begin{minipage}{.25\textwidth}
\sevenspaced
\end{minipage}
&
$X_{5}$: &
\begin{minipage}{.25\textwidth}
\sevenspacee
\end{minipage}
&
$X_{6}$: &
\begin{minipage}{.25\textwidth}
\sevenspacef
\end{minipage}
\\
$X_{7}$: &
\begin{minipage}{.25\textwidth}
\sevenspaceg
\end{minipage}
&
$X_{8}$: &
\begin{minipage}{.25\textwidth}
\sevenspaceh
\end{minipage}
&
$X_{9}$: &
\begin{minipage}{.25\textwidth}
\sevenspacei
\end{minipage}
\\
$X_{10}$: &
\begin{minipage}{.25\textwidth}
\sevenspacej
\end{minipage}
&
$X_{11}$: &
\begin{minipage}{.25\textwidth}
\sevenspacek
\end{minipage}
&
$X_{12}$: &
\begin{minipage}{.25\textwidth}
\sevenspacel
\end{minipage}
\\
$X_{13}$: &
\begin{minipage}{.25\textwidth}
\sevenspacem
\end{minipage}
&
$X_{14}$: &
\begin{minipage}{.25\textwidth}
\sevenspacen
\end{minipage}
&
$X_{15}$: &
\begin{minipage}{.25\textwidth}
\sevenspaceo
\end{minipage}
\end{tabular}
\caption{All images of 7 vertices which have no reduction using Lemmas \ref{simplexreduction} or \ref{pointreduction}\label{sevenspacesfig}}
\end{figure}

\begin{thm}\label{appendixthm}
For the images in Figure \ref{sevenspacesfig}, $X_i$ is reducible for all $i>3$.
\end{thm}
\begin{proof}
We require different arguments for each image, but each uses Lemma \ref{pathreduction} or Lemma \ref{nonsurj}.
\begin{itemize}
\item[$i=4:$]
$X_4$ is reducible by Lemma \ref{pathreduction} with paths $p = (0,1,2)$ and $q = (3,4,5)$.
\item[$i=5:$]
$X_5$ is reducible by Lemma \ref{nonsurj}, since the following map is nonsurjective and homotopic to the identity:
\[
f(x)= \begin{cases}
    x+1 \pmod 6 & \text{if } x\in \{0,1,2,3,4,5\}, \\
    2 & \text{if } x=6.
\end{cases}
\]
\item[$i=6:$]
$X_6$ is reducible by Lemma \ref{pathreduction} with paths $p = (1,2)$ and $q = (0,3)$.
\item[$i=7:$]
$X_7$ is reducible by Lemma \ref{nonsurj}, since the following map is nonsurjective and homotopic to the identity:
\[
f(x)= \begin{cases}
    x+1 \pmod 5 & \text{if } x\in \{0,1,2,3,4\}, \\  
    0 & \text{if } x=5
    \\  
    4 & \text{if } x=6.
\end{cases}
\]
\item[$i=8:$]
$X_8$ is reducible by Lemma \ref{nonsurj}, since the following map is nonsurjective and homotopic to the identity:
\[ 
f(x)= \begin{cases}
    x+1 \pmod 5 & \text{if } x \in \{ 0,1,2,3,4\}, \\
    2 & \text{if } x=5,\\
    3 & \text{if } x=6.
\end{cases}
\]
\item[$i=9:$]
$X_9$ is reducible by Lemma \ref{pathreduction} with paths $p = (1,2)$ and $q=(0,3)$.
\item[$i=10:$]
$X_{10}$ is reducible by Lemma \ref{pathreduction} with paths $p = (3,0,6)$ and $q = (2,1,5)$.
\item[$i=11:$]
$X_{11}$ is reducible by Lemma \ref{pathreduction} with paths $p = (0,3)$ and $q=(1,2)$.
\item[$i=12:$]
$X_{12}$ is reducible by Lemma \ref{nonsurj}, since the following map is nonsurjective and homotopic to the identity:
\[
f(x)= \begin{cases}
    x+1 \pmod 5 & \text{if } x \in \{0,1,2,3,4\}, \\
    0 & \text{if } x=5, \\
    1 & \text{if } x=2.
\end{cases}
\]
\item[$i=13:$]
$X_{13}$ is reducible by Lemma \ref{nonsurj}, since the following map is nonsurjective and homotopic to the identity:
\[
f(x)= \begin{cases}
    x+1 \pmod 5 & \text{if } x\in \{0,1,2,3,4\}, \\
    1 & \text{if } x \in \{5,6\}.
\end{cases}
\]
\item[$i=14:$]
$X_{14}$ is reducible by Lemma \ref{pathreduction} with paths $p = (0,1)$ and $q = (6,2)$.
\item[$i=15:$]
$X_{15}$ is reducible by Lemma \ref{nonsurj}, since the following map is nonsurjective and homotopic to the identity:
\[
f(x) = \begin{cases}
x+1 \pmod 5 & \text{if } x \in \{ 0,1,2,3,4\}, \\
3 & \text{if } x = 5, \\
0 & \text{if } x = 6. 
\end{cases}
\]
\end{itemize}
\end{proof}

\bibliographystyle{plain}

\begin{thebibliography}{99}
\bibitem{boxer94} Boxer, L., Digitally continuous functions. Pattern Recognition Letters 15 (1994), 833-839. 

\bibitem{boxer99} Boxer, L., A classical construction for the digital fundamental group. Journal of Mathematical Imaging and Vision 10 (1999) 51-62.

\bibitem{boxer05} Boxer, L., Properties of digital homotopy. Journal of Mathematical Imaging and Vision 22 (2005), 19--26.

\bibitem{boxer11} Boxer, Laurence; Karaca, Ismet; \"Oztel, Ahmet
Topological invariants in digital images. (English summary) J. Math. Sci. Adv. Appl. 11 (2011), no. 2, 109--140. 

\bibitem{bykov99} Bykov, Alexander I.; Zerkalov, Leonid G.; Pineda, Mario A. Rodriguez. Index of a point of 3-D digital binary image and algorithm for computing its Euler characteristic. Pattern Recognition 32 (1999) 845-850.

\bibitem{chen01} Chen, B., Yau, S-T., and Yeh, Y-N. Graph homotopy and Graham homotopy. Discrete Math., 241 (2001),153--170,

\bibitem{diestel10} Diestel, R. ``Graph Theory''. Springer Verlag. 2010.

\bibitem{dirac61} Dirac, G. A. On rigid circuit graphs. Abhandlungen aus dem Mathematischen Seminar der Universit\"at Hamburg 25 (1961) 71--76. 

\bibitem{ege14} Ege, Ozgur; Karaca, Ismet; Erden Ege, Meltem
Relative homology groups in digital images.
Appl. Math. Inf. Sci. 8 (2014), no. 5, 2337--2345. 

\bibitem{gray71} Gray, Stephen B. Local properties of binary images in two dimensions. IEEE Transactions on computers, vol C-20 (1971) no 5, 551--561.

\bibitem{gross65} Fulkerson, D. R.; Gross, O. A. Incidence matrices and interval graphs. Pacific Journal of Math. 15 (1965) 835--855.


\bibitem{knill12} Josellis, Frank; Knill, Oliver. The Lusternik-Schnirelmann theorem for graphs. \url{http://arxiv.org/abs/1211.0750v2}.

\bibitem{khalimsky90} Khalimsky, Efim; Kopperman, Ralph; Meyer, Paul R. Computer graphics and connected topologies on finite ordered sets. Topology Appl. 36 (1990), no. 1, 1--17. 

\bibitem{rosenfield73} Rosenfield, A. Arcs and Curves in Digital Pictures. Journal of the ACM 20.1 (1973): 81--87. Web.

\bibitem{slapal13} \u{S}lapal, Josef. Topological structuring of the digital plane. 
Discrete Math. Theor. Comput. Sci. 15 (2013), no. 2, 165--176. 
\end{thebibliography}

\end{document}